\documentclass[11pt, reqno]{amsart}
\usepackage{amssymb,amsmath,amscd, amsfonts,latexsym}
\usepackage[top = 4cm, bottom = 4cm, left = 3cm, right = 3cm]{geometry}
\usepackage{hyperref}

\newtheorem{theorem}{Theorem}
\newtheorem{corollary}[theorem]{Corollary}

\newtheorem{lemma}{Lemma}

\theoremstyle{remark}

\title[\title{On the  exponential Diophantine equation ...}]{On the  exponential Diophantine equation related to powers of two consecutive terms of Lucas sequences}

\author[M. Ddamulira and F. Luca]{Mahadi Ddamulira and  Florian Luca}



\address{Mahadi Ddamulira \newline
         \indent Institute of Analysis and Number Theory, Graz University of Technology, \newline
        \indent Kopernikusgasse 24/II, \newline
         \indent A-8010 Graz, Austria}
\address{Max Planck Institute for Mathematics \newline
\indent Bonn, Germany}

\email{mddamulira\char'100tugraz.at; mahadi\char'100aims.edu.gh}

\address{Florian Luca \newline
         \indent School of Mathematics, University of the Witwatersrand, \newline
        \indent Johannesburg, South Africa}
         
\address{Research Group in Algebraic Structures and Applications, King Abdulaziz University,\newline 
         \indent Jeddah, Saudi Arabia}

\address{Max Planck Institute for Mathematics \newline
\indent Bonn, Germany}           

\address{Centro de Ciencias Matem\'aticas UNAM\newline
\indent Morelia, Mexico}

\email{Florian.Luca\char'100wits.ac.za}

\begin{document}


\begin{abstract}
Let  $r\ge 1$ be an integer and ${\bf U}:=(U_{n})_{n\ge 0} $ be the Lucas sequence given by $U_0=0$, $U_1=1, $ and $U_{n+2}=rU_{n+1}+U_n$, for all $ n\ge 0 $. 
In this paper, we show that there are no positive integers $r\ge 3,~x\ne 2,~n\ge 1$ such that $U_n^x+U_{n+1}^x$ is a member of ${\bf U}$. 
\end{abstract}

\maketitle

\section{Introduction}
Let $r\ge 1$ be an integer and ${\bf U}:=(U_{n})_{n\ge 0} $ be the Lucas sequence given by $ U_0=0, ~ U_1=1, $ and 
\begin{align}
U_{n+2}=rU_{n+1}+U_n
\end{align} for all $ n\ge 0$. When $r=1$ ${\bf U}$ coincides with the Fibonacci sequence while when $r=2$ ${\bf U}$ coincides with the Pell sequence. It is well-known that 
\begin{eqnarray}\label{Prob1}
U_n^2+U_{n+1}^2 = U_{2n+1} \quad \text{for all} \quad n\ge 0.
\end{eqnarray}
In particular, the identity \eqref{Prob1} tells us that the sum of the squares of two consecutive terms of ${\bf U}$ is also a term of ${\bf U}$.
When $r=1$, we even have $U_n+U_{n+1}=U_{n+2}$ for all $n\ge 0$ since ${\bf U}$ is the Fibonacci sequence. 
We thus consider the Diophantine equation
\begin{eqnarray}\label{Prob2}
U_n^x+U_{n+1}^x=U_m
\end{eqnarray}
in nonnegative integers $(n,m,x)$ which by Eq. \eqref{Prob1} has the parametric solution $m=2n+1$ when $x=2$ for any $r\ge 1$ and even the parametric solution $m=n+2$ when 
$x=1$ if $r=1$. For $r=1$ Luca and Oyono \cite{Luca1} proved that Eq. \eqref{Prob2} has no positive integer solutions $ (n,m,x)  $ with $n\ge 2 $ and $x\ge 3 $. Rihane et al. \cite{Luca2} studied Eq. \eqref{Prob2} when $r=2$ and proved that there is no positive integer solution $(n,m,x)$  of it with $x\ne 2$. In the same spirit, G\'omez Ruiz and Luca \cite{Gomez} studied Eq. \eqref{Prob2} with $ {\bf U} = \{F_n^{(k)}\}_{n\ge -(k-2)} $, which is the $k-$generalized Fibonacci sequence of recurrence
\begin{align*}
F_n^{(k)}= F_{n-1}^{(k)}+F_{n-2}^{(k)}+\cdots+F_{n-k}^{(k)} \quad \text{for all} \quad n\ge 2
\end{align*}
with the initial conditions $ F_{-(k-2)}^{(k)} =F_{-(k-2)}^{(k)}=\cdots=F_{0}^{(k)} =0$ and $F_1^{(k)}=1$. When $k=2$, this sequence coincides with the sequence of 
Fibonacci numbers. They proved that Eq. \eqref{Prob2} has no positive integer solution $(k, n, m, x) $ with $ k\ge 3 $, $n\ge 2$, and $x\ge 1$. Another related result involving the balancing numbers was studied by Rihane et al. in \cite{Luca3}.

In this paper, we study Eq. \eqref{Prob2} in nonnegative integers $(r,n,m,x)$ treating $r$ as an integer parameter. We may assume that $r\ge 3$ since the cases $r\in \{1,2\}$ have been treated already in \cite{Luca1} and \cite{Luca2}, respectively. The solution with $(n,m)=(0,1)$ (for any $r$ and $x$) is obvious so we omit it and suppose that $n$ is positive. Our main result is the following.
\begin{theorem}\label{ProbM}
There is no positive integer solution $(r,n,m,x)$ of Diophantine equation \eqref{Prob2} with $r\ge 3$ and $x\ne 2$. 
\end{theorem}

\section{Preliminary Results}
\subsection{The Lucas sequence}
Let $$ (\alpha, \beta):= \left(\dfrac{r+\sqrt{r^2+4}}{2},  \dfrac{r-\sqrt{r^2+4}}{2}\right), $$
be the roots of the characteristic equation $ x^2-rx-1=0 $ of the Lucas sequence ${\bf U}=(U_n)_{n\ge 0}$. We put $\Delta=r^2+4=(\alpha-\beta)^2$
for the discriminant of the above quadratic equation. The Binet formula for the general term of ${\bf U}$ is given by

\begin{eqnarray}\label{Prob3}
U_n := \dfrac{\alpha^{n}-\beta^{n}}{\alpha-\beta} \quad \text{for all} \quad n\ge 0.
\end{eqnarray}
One may prove by induction that the inequality
\begin{eqnarray}\label{Prob4}
\alpha^{n-2}\le U_n \le \alpha^{n-1}
\end{eqnarray}
holds for all positive integers $n$.  It is also easy to show that the inequality
\begin{eqnarray}\label{Prob5}
\dfrac{U_n}{U_{n+1}}<\dfrac{1}{r}
\end{eqnarray}
holds for all $n\ge 2$. Indeed, it follows from $U_{n+1}=rU_n+U_{n-1}>rU_n$ for $n\ge 2$. At one point of the argument we will need the companion Lucas sequence 
${\bf V}:=\{V_n\}_{n\ge 0}$ given by $V_0=2,~V_1=r$, and $V_{n+2}=rV_{n+1}+V_n$ for all $n\ge 0$. Its Binet formula is 
\begin{equation}
\label{eq:BinetV}
V_n=\alpha^n+\beta^n\quad {\text{\rm for~all}}\quad n\ge 0.
\end{equation}
There are many relations between members of ${\bf U}$ and ${\bf V}$ such as 
\begin{eqnarray}
U_{2n} & = & U_nV_n;\label{double}\\
U_{n+1}^2-U_nU_{n+2} & = & (-1)^n;\label{squares}\\
V_n^2-\Delta U_n^2 & = & 4(-1)^n;\label{Pell}\\
2U_{m+n} & = & U_mV_n+U_nV_m;\label{addition}\\
\gcd(U_n,U_m) & = & U_{\gcd(m,n)}.\label{gcd}
\end{eqnarray}
We record another one. All such identities follow easily from the  Binet formulas Eq. \eqref{Prob3} and Eq. \eqref{eq:BinetV} of ${\bf U}$ and ${\bf V}$ respectively. 

\begin{lemma}
\label{lem:1}
If $n$ is odd then
\begin{equation}
\label{eq:Unminus1}
U_n-1=\left\{\begin{matrix} U_{(n-1)/2}V_{(n+1)/2} & {\text{if}} & n\equiv 1\pmod 4;\\
U_{(n+1)/2}V_{(n-1)/2} & {\text{if}} & n\equiv 3\pmod 4.\end{matrix}\right.
\end{equation}
\end{lemma}

Both sequences ${\bf U}$ and ${\bf V}$ can be extended to negative indices either by allowing $n$ to be negative in the Binet formulas Eq. \eqref{Prob3} and \eqref{eq:BinetV}
or simply by using the recurrence relation to extend ${\bf U}$ to negative indices via $U_{-n}=-rU_{-(n-1)}+U_{-(n-2)}$ for all $n\ge 1$. The same applies to ${\bf V}$. All the above formulas 
given in Eq. \eqref{double}, \eqref{squares}, \eqref{Pell}, \eqref{addition}, and \eqref{eq:Unminus1} hold when the indices are arbitrary integers, not necessarily nonnegative.

The following lemma is useful. For further details we refer the reader to the book of Koshy \cite{TK}.
\begin{lemma}\label{Fib1}
Let $ \{U_n(r)\}_{n\ge 0} \subseteq {\bf Z}[r]$ be the sequence of polynomials defined by $ U_0(r)=0 $, $ U_1(r)=1 $, and 
\begin{align*}
U_{n+2}(r)= rU_{n+1}(r)+U_{n}(r), \quad \text{for all } \quad n\ge 0.
\end{align*}
Then,
\begin{align}\label{Fib2}
 U_{n}(r) = \sum_{\substack{0\le k\le n  \\ k\not \equiv n \pmod 2}}\binom{\frac{n+k-1}{2}}{k}r^k.
\end{align}
\end{lemma}

Note that the summation range above is only over these $k\in [0,n]$ which have different parity than $n$. We record the following easy but useful consequence of the formula Eq. \eqref{Fib2}.

\begin{lemma}
\label{lem:2}
We have $r\mid U_n$ if $n$ is even and $r\mid U_{n}-1$ if $n$ is odd. 
\end{lemma}

\subsection{Logarithmic height}
Let $\eta$ be an algebraic number of degree $d$ with minimal primitive polynomial over the integers
$$
a_0x^{d}+ a_1x^{d-1}+\cdots+a_d = a_0\prod_{i=1}^{d}(x-\eta^{(i)}),
$$
where the leading coefficient $a_0$ is positive and the $\eta^{(i)}$'s are the conjugates of $\eta$. The \textit{logarithmic height} of $\eta$ is given by
$$ 
h(\eta):=\dfrac{1}{d}\left( \log a_0 + \sum_{i=1}^{d}\log\left(\max\{|\eta^{(i)}|, 1\}\right)\right).
$$
In particular, if $\eta=p/q$ is a rational number with $\gcd (p,q)=1$ and $q>0$, then $h(\eta)=\log\max\{|p|, q\}$. The following are some of the properties of the logarithmic height function $h(\cdot)$, which will be used in the next sections of this paper without reference:
\begin{equation}
\label{eq:heights}
\begin{aligned}
h(\eta\pm \gamma) &\leq h(\eta) +h(\gamma) +\log 2,\\
h(\eta\gamma^{\pm 1})&\leq  h(\eta) + h(\gamma),\\
h(\eta^{s}) &= |s|h(\eta) \quad (s\in\mathbb{Z}). 
\end{aligned}
\end{equation}

\subsection{Linear forms in logarithms and continued fractions}
The following result on linear forms in three logarithms is due to Mignotte \cite{Mignotte}. The result is more general (i.e., the conditions on the parameters involved are somewhat more general), but we will only quote it in the form that we need.
\begin{theorem}\label{Mignotte1}
Consider three algebraic numbers $ \gamma_1, \gamma_2 $, and $ \gamma_3 $, which are all real, greater than $ 1 $ and multiplicatively independent. Put
\begin{align*}
\mathcal{D}:=[\mathbb{Q}(\gamma_1, \gamma_2, \gamma_3 ): \mathbb{Q}].
\end{align*}
Let $ b_1, b_2, b_3 $ be coprime positive integers and consider 
\begin{align*}
\Gamma:=b_2\log\gamma_2 - b_1\log\gamma_1-b_3\log\gamma_3.
\end{align*}
Put
\begin{align*}
d_1:=\gcd (b_1, b_2) = \frac{b_1}{b_1^{'}}=\frac{b_2}{b_2^{'}}, \quad d_3:=\gcd(b_3, b_2)=\frac{b_2}{b_2^{''}}=\frac{b_3}{b_3^{''}}.
\end{align*}
Let $ A_1, A_2 $, and $ A_3 $ be real numbers such that
\begin{align*}
A_i \ge \max\{4, ~4.296 \log\gamma_i+ 2\mathcal{D}h(\gamma_i)\},\quad  i=1,2,3, \quad \text{and} \quad  \Omega:=A_1A_2A_3\ge 100.
\end{align*}
Put
\begin{align*}
b^{'}: =\left(\dfrac{b_1^{'}}{A_2}+\dfrac{b_2^{'}}{A_1}\right)\left(\dfrac{b_3^{''}}{A_2}+\dfrac{b_2^{''}}{A_3}\right) \quad \text{and}\quad  \log \mathcal{B}:= \max\left\{0.882+\log b^{'}, \frac{10}{\mathcal{D}}\right\}.
\end{align*}
Then, either
\begin{align*}
\log|\Gamma|>-790.95\Omega\mathcal{D}^2(\log \mathcal{B})^2,
\end{align*}
or one of the following conditions holds:
\begin{itemize}
\item[(i)] there exist two positive integers $ r_0 $ and $ s_0 $ such that 
\begin{align*}
r_0b_2=s_0b_1
\end{align*}
with
\begin{align*}
r_0\le 5.61A_2 (\mathcal{D}\log\mathcal{D})^{\frac{1}{3}} \quad \text{and} \quad s_0\le 5.61A_1 (\mathcal{D}\log\mathcal{D})^{\frac{1}{3}};
\end{align*}
\item[(ii)] there exist integers $ r_1, s_1, t_1, $ and $ t_2 $, with $ r_1s_1\neq 0 $, such that
\begin{align*}
(t_1b_1+r_1b_3)s_1 = r_1b_2t_2, \quad \gcd (r_1, t_1)=\gcd(s_1, t_2) =1
\end{align*}
which also satisfy
\begin{align*}
|r_1s_1|\le 5.61\delta A_3 (\mathcal{D}\log\mathcal{D})^{\frac{1}{3}}, \quad |s_1t_1|\le 5.61\delta A_1 (\mathcal{D}\log\mathcal{D})^{\frac{1}{3}}, \quad |r_1t_2| \le 5.61\delta A_2 (\mathcal{D}\log\mathcal{D})^{\frac{1}{3}},
\end{align*}
where
\begin{align*}
\delta := \gcd(r_1, s_1).
\end{align*}
\end{itemize}
Moreover, when $ t_1=0 $ we can take $ r_1=1 $, and when $ t_2=0 $ we can take $ s_1 =1$.
\end{theorem}

At some point we will need to treat linear forms in two logarithms of algebraic numbers. To set the stage, let $\gamma_1$ and  $\gamma_2$ be real algebraic numbers which are positive and let ${\mathcal D}:=[{\mathbb Q}(\gamma_1,\gamma_2):{\mathbb Q}]$.
Let $b_1$, $b_2$ be nonzero integers, let $B_1$ and $B_2$ be real numbers larger than $1$ such that
\begin{eqnarray*}
\log B_i\geq \max\left\{h(\gamma_i), \dfrac{|\log\gamma_i|}{\mathcal D}, \dfrac{1}{\mathcal D}\right\},\qquad {\text{\rm for}}\quad  i=1,2,
\end{eqnarray*}
and put
\begin{eqnarray*}
b^{\prime}=\dfrac{|b_{1}|}{{\mathcal D}\log B_2}+\dfrac{|b_2|}{{\mathcal D}\log B_1}.
\end{eqnarray*}
Let
\begin{equation}
\label{eq:Gamma}
\Gamma:= b_1\log\gamma_1+b_2\log\gamma_2.
\end{equation}
The following result of Laurent et al. is Corollary 2 in \cite{Laurent:1995}.

\begin{theorem}\label{Mignotte2}
With the above notations assuming furthermore that $\gamma_{1} $ and $\gamma_{2}$ are multiplicatively independent we have
\begin{eqnarray}
\log |\Gamma|> -24.34{\mathcal D}^4\left(\max\left\{\log b^{\prime}+0.14, \dfrac{21}{\mathcal D}, \dfrac{1}{2}\right\}\right)^{2}\log B_1\log B_2.
\end{eqnarray}
\end{theorem}
Note that the fact that $\Gamma\ne 0$ is already guaranteed by the condition that $\gamma_1$ and $\gamma_2$ are multiplicatively independent together with the fact that $b_1,~b_2$ are nonzero integers.

During the calculations we get upper bounds on our variables which are too large, thus we need to reduce them. To do so we use some results from the theory of continued fractions. 

For the treatment of linear forms homogeneous in two integer variables we use a well-known classical result in the theory of Diophantine approximation due  to Legendre.
\begin{lemma}\label{Legendre}
Let $\tau$ be an irrational number,  $ {p_0}/{q_0}, {p_1}/{q_1}, {p_2}/{q_2}, \ldots $ be  the sequence of convergents of the continued fraction expansion of $ \tau $ and $ M $ be a positive integer. Let $ N $ be a nonnegative integer such that $ q_N> M $. Putting $ a(M):=\max\{a_{i}: i=0, 1, 2, \ldots, N\} $ the inequality
\begin{eqnarray*}
\left|\tau - \dfrac{u}{v}\right|> \dfrac{1}{(a(M)+2)v^{2}},
\end{eqnarray*}
holds for all pairs $ (u,v) $ of positive integers with $ 0<v<M $. Furthermore, if 
$$
\left|\tau-\frac{u}{v}\right|<\frac{1}{2v^2},
$$
then $u/v=p_k/q_k$ for some $k\ge 0$. 
\end{lemma}

For a nonhomogeneous linear form in two integer variables we use a slight variation of a result due to Dujella and Peth{\H o} (see \cite{dujella98}, Lemma 5a). For a real number $X$, we write  $\Vert X \Vert := \min\{|X-n|: n\in\mathbb{Z}\}$ for the distance from $X$ to the nearest integer.
\begin{lemma}\label{Dujjella}
Let $M$ be a positive integer, ${p}/{q}$ be a convergent of the continued fraction expansion of the irrational number $\tau$ such that $q>6M$, and  $A,B,\mu$ be some real numbers with $A>0$ and $B>1$. Furthermore, let $\varepsilon: = \Vert \mu q \Vert -M\Vert\tau q\Vert $. If $ \varepsilon > 0 $, then there is no solution to the inequality
$$
0<|u\tau-v+\mu|<AB^{-w}
$$
in positive integers $u,v$, and $w$ with
$$ 
u\le M \quad {\text{and}}\quad w\ge \dfrac{\log(Aq/\varepsilon)}{\log B}.
$$
\end{lemma}

\section{Proof of Theorem \ref{ProbM}}

\subsection{The cases $n=1$ or $x=1$}
We assume that $n\ge 1$, as the solution with $ n=0 $ is trivial. Since $ U_{n+1}< U_{n+1}+U_n< U_{n+2} $, it follows that the Diophantine equation Eq. \eqref{Prob2} has no solution with $x=1$. Let us assume that $n=1$.  We then get that 
\begin{eqnarray}\label{chu1}
U_m=1+r^{x}.
\end{eqnarray}  
In particular, $U_m\equiv 1\pmod r$. Lemma \ref{lem:2} shows that $m\equiv 1\pmod 2$ and now Lemma \ref{lem:1} shows that
$$
r^x=U_m-1=U_{(m-\delta)/2}V_{(m+ \delta)/2} \quad \text{where} \quad \delta \in \{\pm 1\},\quad \delta \equiv m \pmod 4.
$$
We now recall the Primitive Divisor Theorem of Carmichael \cite{Car} (see \cite{BHV} for the most general statement) for the sequence ${\bf U}$. It states that if 
$\ell>12$, there is a prime factor $p$ of $U_{\ell}$ which is primitive in the sense that $p\nmid U_{k}$ for any positive $k<\ell$. So, assume $m+\delta\ge 14$. Since $m+\delta$ is even we have $U_{m+\delta}=U_{(m+\delta)/2}V_{(m+\delta)/2}$ by the formula Eq. \eqref{double}.  Since $U_{m+\delta}$ has a primitive prime factor $p$, the primitive prime $p$ must be a divisor of $V_{(m+\delta)/2}$, which in turn must divide $r=U_2$, a contradiction. Thus, $m+\delta\le 12$, therefore $m\le 13$. It thus follows that
$$
r^x=U_m-1<\alpha^{m-1}\le \alpha^{12}<(r^2+4)^{6},
$$
so, if $x\ge 13$, then 
$$
3^{x-12}\le r^{x-12}<(1+4/r^2)^6\le (1+4/3^2)^6<9.1,
$$
which gives $x\le 14$. Thus, $m\le 13,~x\le 14$. For each choice of the pair $(m,x)$ with the components in the above ranges, the equation $U_m(r)-1-r^x=0$ is a polynomial equation in $r$. 
After a simple computer search, we found no other solutions to equation Eq. \eqref{chu1} apart from the solution $ (n,m,x) = (1,3,2) $ which has $x=2$ so it is part of the parametric family of solutions indicated at Eq. \eqref{Prob1}.   

So, from now on we assume that $n\ge 2$ and $x\ge 3$.

\subsection{Calculations when $n\in [2,100]$ and $x\in [3,100]$}

Using the equation Eq. \eqref{Prob2} and the inequality Eq. \eqref{Prob4}, we get
\begin{eqnarray*}
\alpha^{(n-1)x}< U_{n+1}^x< U_n^x+U_{n+1}^x = U_m \le  \alpha^{m-1},
\end{eqnarray*}
and 
\begin{eqnarray*}
\alpha^{m-2}< U_m=  U_n^x+U_{n+1}^x  < (U_n+U_{n+1})^{x}< U_{n+2}^{x}<  \alpha^{(n+1)x}.
\end{eqnarray*}
From the above inequalities, we get the following result which we record for future reference.

\begin{lemma}
\label{mversusnx}
The inequalities
\begin{eqnarray}\label{Prob6}
(n-1)x+1< m< (n+1)x+2
\end{eqnarray}
hold for all $x\ge 3$ and $n\ge 2$.
\end{lemma}

We next consider Eq. \eqref{Fib2} given in Lemma \ref{Fib1}. We write Eq. \eqref{Prob2} as
\begin{align}\label{papp1}
\left(\sum_{\substack{0\le k\le n  \\ k\not \equiv n (\text{mod}~ 2)}}\binom{\frac{n+k-1}{2}}{k}r^k\right)^{x} + \left(\sum_{\substack{0\le k\le n+1  \\ k\not \equiv n+1 (\text{mod}~ 2)}}\binom{\frac{n+k}{2}}{k}r^k\right)^{x} = \sum_{\substack{0\le k\le m  \\ k\not \equiv m (\text{mod}~ 2)}}\binom{\frac{m+k-1}{2}}{k}r^k.
\end{align}
Assume first that $ n $ is even. Then Eq. \eqref{papp1} becomes
\begin{align*}
\left(\frac{n}{2}r+\binom{\frac{n+2}{2}}{3}r^3+ \cdots\right)^{x}+\left(1+\binom{\frac{n+2}{2}}{2}r^2+\binom{\frac{n+4}{2}}{4}r^4+ \cdots\right)^{x} = 1+\binom{\frac{m+1}{2}}{2}r^2+\binom{\frac{m+3}{2}}{3}r^4+\cdots,
\end{align*}
which is equivalent to
\begin{align}
\binom{\frac{n+2}{2}}{2}r^2x+\binom{\frac{n+4}{2}}{4}r^4x+\cdots \equiv \binom{\frac{m+1}{2}}{2}r^2+\binom{\frac{m+3}{2}}{3}r^4+\cdots (\text{mod}~ r^x).
\end{align}
The above relation implies that
\begin{align}\label{pip1}
r^{\min\{x,4\}-2}\Big| \binom{\frac{m+1}{2}}{2}-x\binom{\frac{n+2}{2}}{2}.
\end{align}
Similarly, when $ n $ is odd, one is led to the analogous divisibility relation
\begin{align}\label{pip2}
r^{\min\{x,4\}-2}\Big|\binom{\frac{m+1}{2}}{2}-x\binom{\frac{n+1}{2}}{2}.
\end{align}
So, fixing $n\in [2,100]$ and $x\in [3,100]$, inequalities Eq. \eqref{Prob6} give some range for $m$. For each $(n,x,m)$, divisibility relations 
Eq. \eqref{pip1} and Eq. \eqref{pip2} (according to whether $n$ is even or odd) give us some possibilities for $r\ge 3$ and now one checks whether relation Eq. \eqref{Prob2}
holds for this candidate $(n,x,m,r)$.  A computer search with Mathematica in this range for $n$ and $x$ which ran for a few hours  found no solutions. For the search we didn't actually checked 
that formula Eq. \eqref{papp1} holds but we checked that Eq. \eqref{papp1} does not hold modulo $T$, where $T$ is the product of the first $20$ primes. The Mathematica function {\tt powermod} 
allowed us to compute the powers of $r$ modulo $T$ arising from the binomial formula rather quickly. 

\medskip

From now on, we assume that $n\ge 2,~x\ge 3$ and $\max\{n,x\}> 100$.

\subsection{A small linear form in three logs}
We rewrite equation Eq. \eqref{Prob2} as
\begin{eqnarray}
\dfrac{\alpha^{m}}{\alpha-\beta}-U_{n+1}^x = U_n^x+ \dfrac{\beta^{m}}{\alpha-\beta}.
\end{eqnarray}
Dividing both sides of the above equation by $ U_{n+1}^x $ and  using the inequality Eq. \eqref{Prob5}, we obtain
\begin{eqnarray}\label{Prob7}
\left|\alpha^{m}(\alpha-\beta)^{-1}U_{n+1}^{-x}-1\right|=\left(\dfrac{U_n}{U_{n+1}}\right)^{x}+\dfrac{\beta^{m}}{(\alpha-\beta)U_{n+1}^{x}}<2\left(\dfrac{U_n}{U_{n+1}}\right)^{x}<\dfrac{2}{r^{x}}.
\end{eqnarray}

Put
\begin{eqnarray}\label{Pir1}
\Lambda:=\alpha^{m}(\alpha-\beta)^{-1}U_{n+1}^{-x}-1 \quad \text{and} \quad \Gamma:=m\log\alpha -\log(\alpha-\beta) - x\log U_{n+1}.
\end{eqnarray}
We observe that $ \Lambda=e^{\Gamma}-1 $, where $ \Lambda $ and $ \Gamma $ are given by \eqref{Pir1}. Since $ |\Lambda|\le 2/27 $, we have that $ e^{|\Gamma|}\le 27/25 $ and using the inequality \eqref{Prob7} we obtain
\begin{align}\label{Ppx}
|\Gamma|=\left|m\log\alpha -\log (\sqrt{r^2+4})-x\log U_{n+1}\right|\le e^{|\Gamma|}|e^{|\Gamma|}-1|\le  \frac{27|\Lambda|}{25}< \dfrac{2.2}{r^{x}}.
\end{align}
We record the above inequality for future reference.

\begin{lemma}
With $\Gamma$ given by formula Eq. \eqref{Pir1}, inequality Eq. \eqref{Ppx} holds.
\end{lemma}

We want to apply Theorem \ref{Mignotte1} with the following data:
\begin{eqnarray*}
\quad \gamma_1:=\alpha-\beta = \sqrt{r^2+4}, \quad \gamma_2:=\alpha , \quad \gamma_3:=U_{n+1}, \quad b_1:=1, \quad b_2:=m, \quad b_3:=x.
\end{eqnarray*}
We need to check that $ \gamma_1 $, $ \gamma_2 $, and $ \gamma_3 $ are multiplicatively independent. This we do in the next subsection.

\subsection{Checking that $\gamma_1,~\gamma_2,~\gamma_3$ are multiplicatively independent}\label{submult}
Well, assume they are not and let $i,j,k$ be integers not all zero such that 
$$
\gamma_1^{i}\gamma_2^j\gamma_3^k=1.
$$
Squaring and rearranging the above relation we get $\gamma_2^{2j}=(\gamma_1^2)^{-i} \gamma_3^{-2k}\in {\mathbb Q}$. However, $\gamma_2^{2j}$ is also a unit, so an algebraic integer whose reciprocal is also an algebraic integer, and it is also positive, so it must be $1$. Thus, $j=0$. It now follows that 
$i$ and $k$ are both nonzero (since if one of them is, so is the other one) and further $\gamma_3=\gamma_1^{-i/k}$. In particular, all prime factors of $U_{n+1}$ are prime factors of $\Delta:=r^2+4$. But this is also contemplated by the Primitive Divisor Theorem of Carmichael since primes dividing $\Delta$ are not considered
primitive. In particular, $U_{n+1}$ does not have primitive prime factors so $n+1\le 12$. In fact, Theorem C in \cite{BHV} together with Table 1 there show that either $n+1\in \{2,3,4,6\}$ or $n+1\in \{5,12\}$ 
but in this last case, the only such Lucas sequence ${\bf U}$ for which either one of $U_5$ or $U_{12}$ does not have primitive prime factors is the sequence of Fibonacci numbers, which is not our case. Thus, $n+1\in \{2,3,6\}$. Further, for each prime $p$ let $z(p)$ be the index of appearance of $p$ in ${\bf U}$ defined as the smallest positive integer $k$ such that $p\mid U_k$. This always exists for our sequence ${\bf U}$ since $\alpha$ 
is a quadratic unit. It has the additional  property that if $\ell$ is a positive integer then $p\mid U_{\ell}$ if and only if $z(p)\mid \ell$. It is also well--known and easy to prove that if $p\mid \Delta$, then $z(p)=p$. Since also 
$z(p)\mid n+1$ and $n+1\in \{2,3,6\}$, it follows that the only possibilities for $p$ are $p=2,3$. Hence, $r^2+4=2^a 3^b$. However, $3$ cannot divide $r^2+4$ for any positive integer $r$ (because $-4$ is not a quasartic residue modulo $3$), so $b=0$ and $2^a=r^2+4$. Thus, $r=2r_0$ is even, $a\ge 3$ and the equation simplifies to $2^{a-2}=r_0^2+1$. Hence, $r_0$ is odd, so $r_0^2\equiv 1\pmod 8$, therefore $2\| r_0^2+1$, which leads to $a=3,~r_0=1$, which gives $r=2$, is not our case. Hence, indeed $\gamma_1,~\gamma_2,~\gamma_3$ are multiplicatively independent. 

\subsection{Applying Theorem \ref{Mignotte1}}

Since $ \gamma_1, \gamma_2, \gamma_3 \in \mathbb{Q}(\alpha) $, we have $ \mathcal{D}=2 $. We also have 
$$
r<\alpha<{\sqrt{r^2+4}}<r+1.
$$
So, we bound the heights of $\alpha$ and ${\sqrt{r^2+4}}$ in terms or $\log(r+1)$. Since
$$
h(\gamma_1)=h(\alpha-\beta) = \dfrac{1}{2}\log (r^2+4)<\log(r+1), \quad h(\gamma_2)=\dfrac{1}{2}\log\alpha<\dfrac{1}{2}\log(r+1),
$$
and 
$$
h(\gamma_3)=\log U_{n+1} < \log \alpha^n=n\log\alpha<n\log(r+1),
$$
we can take
$$
A_1:=8.296\log(r+1), \quad A_2:=6.296\log(r+1), \quad A_3:=8.296n\log(r+1). 
$$
Thus $\Omega: = A_1A_2A_3 >433 n (\log(r+1))^3 >  100$ since $n\ge 2$ and $r\ge 3$.
Then,
\begin{eqnarray}
\label{eq:bprime}
b^{'}&=& \left(\dfrac{1}{6.296\log(r+1)}+\dfrac{m}{8.296\log (r+1)}\right)\left(\dfrac{x}{6.296\log(r+1)}+\dfrac{m}{8.296n\log(r+1)}\right)\nonumber\\
&< & \dfrac{m}{(\log(r+1))^2}\left(\dfrac{1}{6.296}+\dfrac{1}{8.296}\right)\dfrac{x}{\log(r+1)}\left(\dfrac{1}{6.296}+\dfrac{2}{8.296}\right)\nonumber\\
&< & \dfrac{0.12mx}{(\log(r+1))^2}.
\end{eqnarray}
In the above chain of inequalities, we used the fact that $m<(n+1)x+2<2n x$ (see inequality Eq. \eqref{Prob6}) since $n\ge 2$ and $x\ge 3$. Further, putting 
$$
\log\mathcal{B} :=\max\left\{ 0.882+\log\left(\frac{0.12mx}{(\log(r+1))^2}\right), 5\right\},
$$
we have that either the inequality
\begin{equation}
\label{eq:either}
\log |\Gamma|>-790.95\times 434 n (\log (r+1))^3\times 2^2\times (\log {\mathcal B})^2
\end{equation}
holds, or the other possibilities (i), (ii) from Theorem \ref{Mignotte1} hold. We treat (i) and (ii) later and deal with the above inequality Eq. \eqref{eq:either} at this stage. 
If $\log\mathcal{B}=5$, then 
\begin{equation}
\label{eq:x1}
x<mx<\frac{e^{5-0.882}}{0.12} (\log (r+1))^2 <512 (\log (r+1))^2.
\end{equation}
On the other hand, if $\log \mathcal{B}>5$, then 
\begin{equation}
\label{eq:lowbound}
\log \mathcal{B}:=0.882+\log \left(\frac{0.12mx}{(\log (r+1))^2}\right)<\log\left(\frac{0.3mx}{(\log(r+1))^2}\right),
\end{equation}
where in the above inequality we used the fact that $e^{0.882}\times 0.12<0.3$. 
Thus, we get that
\begin{align*}
\log|\Gamma|&>-790.95\times 434 n (\log(r+1))^3 \times 2^{2}\times \left(\log\left(\dfrac{0.3mx}{(\log(r+1))^2}\right)\right)^{2}\\
&>-1.374\times 10^{6}n (\log (r+1))^3 \left(\log\left(\dfrac{0.3mx}{(\log(r+1))^2}\right)\right)^{2}.
\end{align*}
Comparing this inequality with Eq. \eqref{Ppx}, we get that
\begin{eqnarray*}
x\log r -\log 2.2< 1.374\times 10^{6}n (\log (r+1))^3 \left(\log\left(\dfrac{0.3mx}{(\log(r+1))^2}\right)\right)^{2},
\end{eqnarray*}
which implies, via the inequality 
\begin{equation}
\label{eq:117}
\frac{\log (r+1)}{\log r}=1+\frac{\log(1+1/r)}{\log r}<1+\frac{1}{r\log r},
\end{equation}
that
\begin{eqnarray}\label{Prob8}
x &  < &  \frac{\log 2.2}{\log r}+1.374\times 10^6 n\left(\frac{\log(r+1)}{\log r}\right)(\log (r+1))^2 \left(\log\left(\dfrac{0.3mx}{(\log(r+1))^2}\right)\right)^{2}\nonumber\\
 & < & \frac{\log 2.2}{\log r}+1.374\times 10^6 n \left(1+\frac{1}{r\log r}\right) (\log (r+1))^2 \left(\log\left(\dfrac{0.3mx}{3(\log(r+1))^2}\right)\right)^{2}\nonumber\\
& < & 1.38\times 10^{6}n\left(1+\frac{1}{r\log r}\right) (\log (r+1))^2 \left(\log\left(\dfrac{0.3mx}{(\log (r+1))^2}\right)\right)^{2}.
\end{eqnarray}
Using the inequality  Eq. \eqref{Prob6}, we know that $ m< (n+1)x+2 <(n+1)(x+1)$ (because $n\ge 2$), and substituting this in \eqref{Prob8} we get that
\begin{align}\label{Prob9}
x< 1.38\times 10^{6}n \left(1+\frac{1}{r\log r}\right) (\log (r+1))^2 \left(\log\left(\dfrac{0.3(n+1)(x+1)^{2}}{(\log r)^2}\right)\right)^{2}.
\end{align}
We now turn our attention to the possibilities (i) and (ii). In case (i), there are positive integers $r_0,s_0$, which may be assumed to be coprime, such that $r_0b_2=s_0b_1$.
So, we get $r_0m=s_0$ and since $r_0,s_0$ are coprime, we take $r_0=1, s_0=m$, and we get
\begin{eqnarray*}
m & = & s_0<5.61 A_1(\mathcal{D}\log \mathcal{D})^{1/3}<5.61\times 8.296\times (2\log 2)^{1/3} \log(r+1)\\
& < & 5.61\times 8.296\times (2\log 2)^{1/3}  \log(r+1)\\
& < & 52 \log(r+1).
\end{eqnarray*}
Since $m>(n-1)x+1>x$, this situation gives
\begin{equation}
\label{eq:xi}
x<52\log(r+1).
\end{equation}
This was in situation (i). In situation (ii), we have integers $r_1,s_1,t_1,t_2$ with $r_1s_1\ne 0$ and 
$$
(t_1b_1+r_1b_3)s_1=r_1b_2t_2,\qquad \gcd(r_1,t_1)=\gcd(s_1,t_2)=1.
$$
Thus, for us, we have
$$
(t_1+r_1x)s_1=r_1mt_2,\qquad \gcd(r_1,t_1)=\gcd(s_1,t_2)=1.
$$
Reducing the above equation modulo $r_1$ we get $t_1s_1\equiv 0\pmod {r_1}$ and since $\gcd(t_1,r_1)=1$, we get that $r_1\mid s_1$. So, we put 
$s_1=r_1s_1'$ and simplify both sides of the above equation by $r_1$ to get 
$$
(t_1+r_1x)s_1'=mt_2. 
$$
Consequently, for us $\delta=\gcd(r_1,s_1)=r_1$. Hence,
\begin{eqnarray}
\label{eq:bounds}
|t_1s_1'| & < & 5.61 A_1 (\mathcal {D}\log \mathcal{D})^{1/3}<5.61\times 8.296\log(r+1)\times (2\log 2)^{1/3}<52\log(r+1);\nonumber\\
|t_2| & < & 5.61 A_2 (\mathcal {D}\log \mathcal{D})^{1/3}<5.61\times 6.296 (2\log 2)^{1/3}\log (r+1)<40\log (r+1);\\
|r_1s_1'| & < & 5.61 A_3  (\mathcal {D}\log \mathcal{D})^{1/3}<5.61 \times 8.296 (2\log 2)^{1/3} n\log (r+1)<52 n\log (r+1).\nonumber
\end{eqnarray}
Assume first that $t_2=0$. Then 
$$
x=|t_1|/|r_1|\le |t_1|<52\log(r+1),
$$
which is the same as \eqref{eq:xi}. Assume next that $t_2\ne 0$. We return to inequality Eq. \eqref{Ppx} and multiply both sides by $t_2$ and get
$$
\left|m t_2\log\gamma_2 -t_2\log \gamma_1 -xt_2\log \gamma_3\right|< \dfrac{2.2|t_2|}{r^{x}}.
$$
We substitute $mt_2$ by $t_1s_1'+(r_1s_1')x$ inside the left--hand side above and then the left--hand side above becomes 
\begin{equation}
\label{eq:222}
\left| \log\left(\frac{\gamma_2^{t_1s_1'}}{\gamma_1^{t_2}}\right)+x\log\left(\frac{\gamma_2^{r_1s_1'}}{\gamma_3^{t_2}}\right)\right|<\frac{2.2|t_2|}{r^x}.
\end{equation}
Inequality \eqref{eq:222} is of the form 
\begin{equation}
\label{eq:234}
|\Gamma_1|<\frac{2.2|t_2|}{r^x},\quad {\text{\rm where}}\quad \Gamma_1:=\log \eta_1+x\log \eta_2,
\end{equation}
and
$$
\eta_1:=\frac{\gamma_2^{t_1s_1'}}{\gamma_1^{t_2}},\quad \eta_2:=\frac{\gamma_2^{r_1s_1'}}{\gamma_3^{t_2}}.
$$
We check that $\eta_1$ and $\eta_2$ are multiplicatively independent. If not, there are integers $i,j$ not both zero such that $\eta_1^i\eta_2^j=1$. This gives
$$
\left(\frac{\gamma_2^{t_1s_1'}}{\gamma_1^{t_2}}\right)^i\left(\frac{\gamma_2^{r_1s_1'}}{\gamma_3^{t_2}}\right)^j=1.
$$
If $i\ne 0$, this gives a multiplicative dependence among $\gamma_1,\gamma_2,\gamma_3$ with the exponent of $\gamma_1$ being the nonzero integer $-t_2i$, a contradiction with the main result of Subsection \ref{submult}. 
Thus, $i=0$, so $j\ne 0$, and we get again a multiplicative relation among $\gamma_2,\gamma_3$ (the exponent of $\gamma_3$ being the nonzero integer $-t_2j$), which is the same contradiction. Thus, indeed $\eta_1$ and $\eta_2$ are multiplicatively independent and they are also positive.  So, we  are in position to apply Theorem 
\ref{Mignotte2} to the left--hand side of inequality Eq. \eqref{eq:234}. We compute $\log B_i$ for these choices. We have, by the properties \eqref{eq:heights}, 
\begin{eqnarray*}
h(\eta_1) & \le &  |t_1s_1'|h(\gamma_2)+|t_2|h(\gamma_1)\le (52\log (r+1))(1/2\log (r+1))+(40\log (r+1))\log (r+1)\\
& = & 66 (\log (r+1))^2;\\
h(\eta_2) & \le & |r_1s_1'| h(\gamma_2)+|t_2|h(\gamma_3)\le 52 n\log (r+1) (1/2\log(r+1))+(40\log(r+1))(n\log (r+1))\\
& = & 66n (\log(r+1))^2.
\end{eqnarray*}
Since $|\log \gamma_i|/2\le h(\gamma_i)$ holds for $i=1,2,3$, it follows, by the absolute value inequality, that the same inequalities are satisfied by the numbers $|\log \eta_i|/2$ for $i=1,2$. Thus, since ${\mathcal D}=2$, we can take 
$$
\log {B}_1:=66 (\log (r+1))^2,\quad \log {B}_2:=66 n(\log (r+1))^2.
$$
We bound
$$
\frac{1}{2\log B_2}+\frac{x}{2\log B_1}  =  \frac{1}{132(\log (r+1))^2} \left(\frac{1}{n}+x\right)<\frac{x+1}{132(\log(r+1))^2}.
$$
Hence, we take 
$$
b':=\frac{x+1}{132(\log(r+1))^2}.
$$
Now Theorem \ref{Mignotte2} gives
\begin{eqnarray}
\label{eq:iiii}
\log |\Lambda| & > & -24.34\times 2^4(\max\{\log b'+0.14,10.5\})^2\times (66(\log(r+1))^2)\times (66n (\log(r+1))^2)\nonumber\\
& > & -1.627\times 10^6 n(\log (r+1))^4 M^2,
\end{eqnarray}
where $M:=\max\{\log b'+0.4,10.5\}$. In case $M=10.5$, we get
$$
b'<e^{10.5-0.4}<24400,
$$
which gives
\begin{equation}
\label{eq:xii}
x+1<24400\times 132 (\log(r+1))^2<3.3\times 10^6 (\log(r+1))^2.
\end{equation}
Finally, suppose that 
$$
M=\log b'+0.4=\log(e^{0.4} b')<\log(1.5 b')=\log\left(\frac{x+1}{88(\log(r+1))^2}\right).
$$
Comparing inequality \eqref{eq:iiii} with inequality \eqref{eq:234}, we get
$$
x\log r-\log(2.2|t_2|)<1.627\times 10^6 n (\log(r+1))^4 \left(\log\left(\frac{x+1}{88(\log(r+1))^2}\right)\right)^2.
$$
Since $2.2|t_2|<88\log(r+1)$, we get that 
$$
x<\frac{\log(88\log(r+1))}{\log r}+1.627\times 10^6 \left(\frac{\log(r+1)}{\log r}\right) (\log(r+1))^3 \left(\log\left(\frac{x+1}{88(\log(r+1))^2}\right)\right)^2.
$$
The first summand in the right--hand side is $<5$ for all $r\ge 3$. Using inequality Eq. \eqref{eq:117}, we get that
\begin{equation}
\label{eq:last}
x<1.63\times 10^6 n \left(1+\frac{1}{r\log r}\right)(\log(r+1))^3 \left(\log\left(\frac{x+1}{88(\log(r+1))^2}\right)\right)^2.
\end{equation}
To summarise, either we are in the first situation of Theorem \ref{Mignotte1} and $\log \mathcal{B}=5$, in which case inequality Eq. \eqref{eq:x1} holds, or $\log \mathcal{B}>5$ in which case inequality Eq. \eqref{Prob9} holds, or we are in the exceptional case (i) for which inequality Eq. \eqref{eq:xi} holds, which is contained in inequality Eq. \eqref{eq:x1}, or we are in the exceptional situation (ii) in which case 
either inequality Eq. \eqref{eq:xii} holds, or inequality Eq. \eqref{eq:last} holds. Since inequality Eq. \eqref{eq:x1} is contained in inequality Eq. \eqref{eq:xii}, 
it follows, using the inequality $1/88<0.12$, that we proved the following result.

\begin{lemma}
\label{lem:boundsonx}
One of the following inequalities holds:
\begin{eqnarray}
x & < & 3.3\times 10^6 (\log(r+1))^2;\label{ineq1}\\
x & < & 1.38\times 10^{6}n\left(1+\frac{1}{r\log r}\right) (\log (r+1))^2 \left(\log\left(\dfrac{0.3(n+1)(x+1)^{2}}{(\log(r+1))^2}\right)\right)^{2};\label{ineq2}\\
x & < & 1.63\times 10^6 n\left(1+\frac{1}{r\log r}\right)(\log(r+1))^3 \left(\log\left(\frac{0.12(x+1)}{(\log(r+1))^2}\right)\right)^2.\label{ineq3}
\end{eqnarray}
\end{lemma}

\subsection{More inequalities in terms of $n$ and $x$}

We put 
\begin{equation}
\label{eq:kappa}
 \kappa:=nx+1-m
\end{equation}
Later we shall show that $\kappa$ is positive except possibly if $r=3$. In this section, assuming that it is positive, we show how it gives some lower bounds for $x$ in terms of $n$. 

\begin{lemma}\label{Luca1}
The following holds:
\begin{itemize}
\item[(i)] $\kappa\ne 1$;
\item[(ii)] If $\kappa=2$ and $n\ge 3$ then $x\ge r^{\max\{2,n-3\}}$;
\item[(iii)] If $\kappa\ge 3$, then $\kappa\ge n/2$.
\end{itemize}
\end{lemma}

\begin{proof}
(i). If $\kappa=1$, then $m=nx$. So the equation \eqref{Prob2} becomes
$$
U_n^x+U_{n+1}^x=U_{nx}.
$$
If $p$ is prime dividing $U_n$ (which exists since $n>1$), then $p\mid U_n^x$ and $p\mid U_n\mid U_{nx}$, so from the above equation we get $p\mid U_{n+1}$, a contradiction since $\gcd(U_n,U_{n+1})=U_{\gcd(n,n+1)}=1$ by relation \eqref{gcd}.

\medskip 

(ii) In this case $m=nx-1$ so the equation \eqref{Prob2} becomes 
$$
U_n^x+U_{n+1}^x=U_{nx-1}.
$$
In particular, $U_{nx-1}-U_{n+1}^x\equiv 0\pmod {U_n^2}$. We study this congruence. In what follows for three algebraic integers $a,b,c$, we write $a\equiv b\pmod c$ if $(a-b)/c$ is an algebraic integer. Write 
$$
U_n=\frac{\alpha^n-\beta^n}{\alpha-\beta}\qquad {\text{\rm as}}\qquad \alpha^n=\beta^n+{\sqrt{\Delta}} U_n.
$$
Then 
$$
\alpha^{nx}=(\beta^n+{\sqrt{\Delta}}U_n)^x\equiv \beta^{nx}+x\beta^{n(x-1)} {\sqrt{\Delta}} U_n\pmod {\Delta U_n^2}.
$$
Thus, 
\begin{eqnarray*}
U_{nx-1} & = & \frac{\alpha^{nx}\alpha^{-1}-\beta^{nx-1}}{{\sqrt{\Delta}}}\\
& \equiv  & \frac{(\beta^{nx}+x\beta^{n(x-1)}{\sqrt{\Delta}} U_n)\alpha^{-1}-\beta^{nx-1}}{\alpha-\beta}\pmod {{\sqrt{\Delta}} U_n^2}\\
& \equiv & \frac{\beta^{nx}\alpha^{-1}-\beta^{nx}\beta^{-1}}{\alpha-\beta}+x\beta^{n(x-1)}\alpha^{-1} U_n\pmod {U_n^2}\\
& \equiv & \beta^{nx}+x\beta^{n(x-1)}\alpha^{-1} U_n \pmod {U_n^2}.
\end{eqnarray*}
On the other hand,
\begin{eqnarray*}
U_{n+1}^x & = & \left(\frac{\alpha^{n+1}-\beta^{n+1}}{{\sqrt{\Delta}}}\right)^x\\
& \equiv  & \left(\frac{(\beta^n+{\sqrt{\Delta}} U_n)\alpha-\beta^{n+1}}{\sqrt{\Delta}}\right)^x\pmod {U_n^2}\\
& \equiv & (\beta^n+U_n\alpha )^x\pmod {U_n^2}\\
& \equiv & \beta^{nx}+x\beta^{n(x-1)} \alpha U_n \pmod {U_n^2}.
\end{eqnarray*}
Thus, 
\begin{eqnarray*}
U_{nx-1}-U_{n+1}^x & \equiv &  (\beta^{nx}+x\beta^{n(x-1)}\alpha^{-1} U_n)-(\beta^{nx}+x\beta^{n(x-1)}\alpha U_n)\pmod {U_n^2}\\
&\equiv & x\beta^{n(x-1)}(\alpha^{-1}-\alpha)U_n\pmod {U_n^2}\\
& \equiv & -x\beta^{n(x-1} rU_n\pmod {U_n^2}.
\end{eqnarray*}
In the last step above, we used the fact that $\alpha^{-1}-\alpha=-\beta-\alpha=-r$. Since the expression $U_{nx-1}-U_{n+1}^x$ is divisible by $U_n^2$, we get that $U_n^2\mid \beta^{n(x-1)} xr U_n$. 
Since $\beta$ is a unit, we get that $U_n\mid xr$. For $n=2$, this gives us nothing since $U_2=r$. For $n=3$, $U_3=r^2+1$ is coprime to $r$, so $U_3\mid x$, which gives $x\ge r^2+1>r^2$.
For $n=4$, we have that $U_4=r(r^2+2)$ divides $rx$, so $r^2+2\mid x$ giving $x\ge r^2+2>r^2$. Finally, for $n\ge 5$, we have that $U_n>\alpha^{n-2}>r^{n-2}$ and so
$x\ge U_n/r\ge r^{n-3}$. This proves (ii).

\medskip

(iii) We may assume that $n\ge 7$, otherwise the conclusion is trivial. Recall that $V_n=\alpha^n+\beta^n$. Relations Eq. \eqref{squares} and \eqref{Pell} are 
$$
U_{n+1}^2-U_nU_{n+2}=(-1)^n\qquad {\text{\rm and}}\qquad V_n^2-\Delta U_n^2=4(-1)^n.
$$
In particular, $U_{n+1}^2\equiv (-1)^n\pmod {U_n}$ and $V_n^2\equiv 4(-1)^n\pmod {U_n}$. We also use the fact  that $U_{-m}=(-1)^{m-1}U_m$, $V_{-m}=(-1)^mV_m$ and relation Eq. \eqref{addition} which is
$$
2U_{m+n}=U_mV_n+U_nV_m.
$$
Armed with these facts, writing $m=nx-(\kappa-1)$ and 
$$
U_n^x+U_{n+1}^x=U_{nx-(\kappa-1)},
$$
we multiply both sides of the above equation by $2$ and write
$$
2U_n^x+2U_{n+1}^x=2U_{nx-(\kappa-1)}=U_{nx}V_{-(\kappa-1)}+V_{nx}U_{-(\kappa-1)}.
$$
We  square both sides of the above equation and reduce it modulo $U_n$ taking into account that $U_n\mid U_{nx}$ and $V_{nx}^2\equiv 4(-1)^{nx}\pmod {U_{nx}}\equiv 4(-1)^{nx}\pmod {U_n}$, and get
\begin{eqnarray*}
4 (-1)^{nx}& \equiv & 4(U_{n+1}^2)^x\pmod {U_n}\equiv (U_{nx}V_{-(\kappa-1)}+ V_{nx}U_{-(\kappa-1)})^2\pmod {U_n}\\
& \equiv & V_{nx}^2 U_{-(\kappa-1)}^2\pmod {U_n}\equiv V_{nx}^2 U_{\kappa-1}^2\pmod {U_n}\equiv  4(-1)^{nx}U_{\kappa-1}^2\pmod {U_n}.
\end{eqnarray*}
Thus, $U_n\mid 4(U_{\kappa-1}^2-1)$. The right--hand side is nonzero since $\kappa>2$. If $\kappa-1$ is odd, then 
$$
4(U_{\kappa-1}^2-1)=4U_{\kappa-2}U_{\kappa}
$$
by relation Eq. \eqref{squares} (with $n+1:=\kappa-1$). Since $n\ge 7$ and ${\bf U}$ is not the Fibonacci sequence, it follows that $U_n$ has a primitive divisor $p$, which must divide one of $U_{\kappa-2}$ or $U_{\kappa}$. Thus, $z(p)=n$ 
divides one of $\kappa-2$ or $\kappa$, so we get $\kappa\ge n$, which is a better conclusion than the one desired. If $\kappa-1$ is even and $U_n\mid 2(U_{\kappa-1}^2-1)$, we then get
$$
\alpha^{n-2}< U_n<2U_{\kappa-1}^2<r(\alpha^{\kappa-1})^2<\alpha^{2\kappa-1},
$$
so $2\kappa-2\ge n-2$, therefore $\kappa\ge n/2$. The above argument was based on the fact that $r\ge 2$. In particular, if $r\ge 4$, then the same argument gives again that
$$
\alpha^{n-2}<U_n<4U_{\kappa-1}^2<r(\alpha^{\kappa-1})^2<\alpha^{2\kappa-1},
$$
so $\kappa\ge n/2$. So, the only case when the above arguments fail are when $r=3$ and $4\mid U_n$. It then follows that $n$ is even (in fact, $n$ is a multiple of $6$, but we shall not need that), so $r\mid U_n\mid 4(U_{\kappa-1}^2-1)$.
But $U_{\kappa-1}$ is a multiple of $r$ (since $\kappa-1$ is even), so $U_{\kappa-1}^2-1$ is coprime to $r$. Thus, $r\mid 4$, which is false. This finishes the proof of (iii).
\end{proof}

\begin{corollary}
If $\kappa>0$, then $x>n/2$.
\end{corollary}

\begin{proof}
By Lemma \ref{Luca1}, if $\kappa=2$, then $x\ge r^{\min\{2,n-3\}}\ge 3^{\min\{2,n-3\}}>n/2$ for any $n\ge 2$. If $\kappa\ge 3$, then 
$$
1+nx-m=\kappa\ge n/2,
$$
which leads to $1+nx-n/2\ge m$. Comparing this with the lower bound $m>(n-1)x+1$ given by inequality Eq. \eqref{Prob6}, we get $x>n/2$.
\end{proof}

\subsection{Another inequality among $r,n,m,x$}

In this section, we return to inequality Eq. \eqref{Ppx} and rewrite it in order to deduce a good approximation of $\log r$ by a rational number whose denominator is a multiple of $r^2$. Let's get to work. We need approximations of $\log \alpha$ and $\log {\sqrt{r^2+4}}$ in terms of $\log r$.

\begin{lemma}
\label{lem:approx}
For $r\ge 3$, we have the following approximations:
\begin{eqnarray}
\log {\sqrt{r^2+4}} & = & \log r+\frac{2}{r^2}+\zeta,\quad |\zeta|<\frac{5.81}{r^4};\label{rsquare}\\
\log \alpha & = & \log r+\frac{1}{r^2}+\zeta',\quad |\zeta'|<\frac{3.64}{r^4}.\label{alpha}
\end{eqnarray}
\end{lemma}

\begin{proof}
We have
$$
\log {\sqrt{r^2+4}}=\frac{1}{2}\log(r^2+4)=\log r+\frac{1}{2}\log\left(1+\frac{4}{r^2}\right).
$$
With $z:=4/r^2$, we have $|z|\le 4/9$ and 
$$
\log(1+z)=z+\zeta_1,\quad |\zeta_1|\le \sum_{k\ge 2} \frac{z^k}{k}=\frac{z^2}{2}\left(1+(2/3)z+(2/4)z^2+\cdots\right).
$$
The expression in parenthesis above is smaller than
$$
c_1:=1+(2/3)(4/9)+(2/4)(4/9)^2+\cdots=(-4/9-\log(1-4/9))\times 2\times (9/4)^2.
$$
Hence, 
$$
\log{\sqrt{r^2+4}}=\log r+\frac{2}{r^2}+\frac{\zeta_1}{2}:=\log r+\frac{2}{r^2}+\zeta,\quad |\zeta|=\frac{|\zeta_1|}{2}<\left(\frac{c_1}{4}\right)z^2=\frac{4c_1}{r^4}<\frac{5.81}{r^4}.
$$
For $\alpha$, we write 
$$
\log \alpha=\log r+\log\left(\frac{1}{2}+{\sqrt{\frac{1}{4}+\frac{1}{r^2}}}\right)=\log r+\log(1+z_1),\quad z_1:={\sqrt{\frac{1}{4}+\frac{z}{4}}}-\frac{1}{2}.
$$
Note that $|z_1|\le 1/r^2$. Thus, 
$$
\log(1+z_1)=z_1+\zeta_2, \quad |\zeta_2|\le \frac{|z_1|^2}{2} (1+(2/3)|z_1|+\cdots)\le \frac{c_2}{2r^4},
$$
where by the previous arguments, 
$$
c_2=(-\lambda_0-\log(1-\lambda_0))\times 2 \times \lambda_0^{-2},\quad {\text{\rm with}}\quad \lambda_0:={\sqrt{\frac{1}{4}+\frac{1}{9}}}-\frac{1}{2}.
$$
It remains to expand $z_1$. For this, we have
$$
z_1=\frac{1}{2}({\sqrt{1+z}}-1)=\frac{1}{2}\left(\frac{z}{2}+\zeta_3\right),\quad |\zeta_3|\le \sum_{k\ge 2} \left|\binom{k}{1/2}\right| z^k.
$$
Since $|\binom{k}{1/2}|\le 1/4$ for all $k\ge 1$, it follows that
$$
|\zeta_3|\le \frac{1}{4}\sum_{k\ge 2} z^k=\frac{z^2}{4(1-z)}\le \frac{36}{5r^4}.
$$
Hence,
$$
\log \alpha=\log r+\frac{z}{4}+\left(\zeta_2+\frac{\zeta_3}{2}\right)=:\log r+\frac{1}{r^2}+\zeta',\quad |\zeta'|<\left(\frac{c_2}{2}+\frac{36}{10}\right)\frac{1}{r^2}<\frac{3.64}{r^4}.
$$
\end{proof}

The following estimate is the main result of this section. 

\begin{lemma}
\label{lem:better}
If $r\ge 4$, then $\kappa>0$. Furthermore,
\begin{equation}
\label{eq:lowbound}
x>\frac{\kappa r^2\log r+1}{1+5/r}.
\end{equation}
\end{lemma}

\begin{proof}
We shall use the approximations given in Lemma \ref{lem:approx} but we also need an approximation of $\log U_{n+1}$. We have 
\begin{equation}
\label{eq:Un}
\log U_{n+1}=\log \left(\frac{\alpha^{n+1}}{\alpha-\beta}\left(1-\left(\frac{\beta}{\alpha}\right)^{n+1}\right)\right)=(n+1)\log \alpha-\log ({\sqrt{r^2+4}})+\zeta'',
\end{equation}
where
$$
\zeta''=\log\left(1-\left(\frac{\beta}{\alpha}\right)^{n+1}\right).
$$
Since $\beta=-\alpha^{-1}$, it follows that $|\beta/\alpha|=1/\alpha^2$. Thus, 
\begin{equation}
\label{eq:zeta''}
|\zeta''|\le \frac{1}{\alpha^{2n+2}}\left(1+\sum_{k\ge 1} \frac{1}{k(\alpha^{2n+2})^k}\right)\le \frac{1}{\alpha^{2n+2}}\left(1+\frac{1}{2(1-1/\alpha^{2n+2})}\right)<\frac{1.51}{\alpha^{2n+2}},
\end{equation}
where for the last inequality we used the fact that $\alpha>r\ge 3$ and $n\ge 2$. Inserting estimates Eq. \eqref{rsquare} and Eq. \eqref{alpha} together with Eq. \eqref{eq:Un} into inequality Eq. \eqref{Ppx}, we get
\begin{eqnarray*}
|\Gamma| & = & \left|m\log \alpha-\log({\sqrt{r^2+4}})-x((n+1)\log \alpha-\log({\sqrt{r^2+4}})+\zeta'')\right|\\
& = & \left|(m-x(n+1))\log \alpha+(x-1)\log ({\sqrt{r^2+4}})-x\zeta''\right|\\
& = & \left| (m-x(n+1))\left(\log r+\frac{1}{r^2}+\zeta'\right)+(x-1)\left(\log r+\frac{2}{r^2}+\zeta\right)-x\zeta''\right|\\
& = & \left| (m-nx-1)\log r+\frac{(m-nx-1)+(x-1)}{r^2}+(m-x(n+1))\zeta'+(x-1)\zeta+x\zeta''\right|.
\end{eqnarray*}
We recognise the coefficient of $\log r$ as the number we denoted $-\kappa$ in \eqref{eq:kappa}. Using inequality Eq. \eqref{Ppx}, we get
$$
\left|-\kappa \log r+\frac{-\kappa+(x-1)}{r^2}\right|<\frac{2.2}{r^x}+|m-x(n+1)||\zeta'|+|x-1||\zeta|+x|\zeta''|.
$$
Inequality Eq. \eqref{Prob6} shows that $m-x(n+1)\in [-2x+2,2]$. In particular, $|m-x(n+1)|\le 2(x-1)$. We thus get, by estimates  Eq. \eqref{rsquare} and Eq. \eqref{alpha} together with Eq. \eqref{eq:Un}, that
$$
\left|-\kappa \log r+\frac{-\kappa+(x-1)}{r^2}\right|<\frac{4}{r^x}+\frac{7.28(x-1)}{r^4}+\frac{5.81(x-1)}{r^4}+\frac{{1.51x}}{r^6}.
$$
Since $x\ge 3$ and $r\ge 3$, we get that the last term satisfies
$$
\frac{1.51 x}{r^6}\le \left(\frac{1.51\times (3/2)}{3^2}\right)\frac{(x-1)}{r^4}<\frac{0.26(x-1)}{r^4}.
$$
Hence,
\begin{equation}
\label{eq:ttt}
\left|-\kappa \log r+\frac{-\kappa+(x-1)}{r^2}\right|<\frac{2.2}{r^x}+\frac{13.35(x-1)}{r^4}.
\end{equation}
Assume that $\kappa\le 0$. We then  get that
$$
\frac{x-1}{r^2}\le \frac{2.2}{r^x}+\frac{13.35(x-1)}{r^4},
$$
which implies that
$$
1\le \frac{13.35}{r^2}+\frac{2.2}{r^{x-2}(x-1)}.
$$
If $r\ge 5$, the right hand side is 
$$
\le \frac{13.35}{25}+\frac{2.2}{5\times 4}=\frac{13.54}{16}<1,
$$
a contradiction. Similarly, if $r\ge 4$ and $x\ge 4$, then the right side is 
$$
\le \frac{13.25}{16}+\frac{2.2}{16\times 3}<\frac{14}{16}<1.
$$
Thus, if $r\ge 5$ or $r=4$ and $x\ge 4$, then $\kappa>0$ and now Lemma \ref{Luca1} applies. We will show at the end of this proof that $\kappa\ge 0$ for 
$(r,x)=(4,3)$ as well. Multiplying both sides of estimate \eqref{eq:ttt} by $r^2$, we get
$$
|-\kappa(r^2\log r+1)+(x-1)|<(x-1)\left(\frac{13.35}{r^2}+\frac{2.2}{r^{x-2}(x-1)}\right).
$$
Hence, 
$$
\kappa(r^2\log r+1)<(x-1)\left(1+\left(\frac{13.35}{r^2}+\frac{2.2}{r^{x-2}(x-1)}\right)\right)\le (x-1)\left(1+\frac{5}{r}\right),
$$
which gives estimate Eq. \eqref{eq:lowbound}. It remains to treat the case $(r,x)=(4,3)$. By inequality \eqref{Prob6}, we have $m<3(n+1)+1=3n+4$ so $\kappa=3n+1-m\ge -2$. 
So, the only instances in which $\kappa\le 0$ is possible  are when $m=3n+3,~m=3n+2,~3n+1$. Well, let us show that this is not possible by proving that 
$$
U_n^3+U_{n+1}^3<U_{3n+1}.
$$
Using the Binet formula Eq. \eqref{Prob3}, this is implied by 
$$
\alpha^{-1}\left(1+\frac{1}{\alpha^{6n}}\right)^3+\alpha^2\left(1+\frac{1}{\alpha^{6n+6}}\right)^3<\Delta\left(1-\frac{1}{\alpha^{6n+3}}\right),
$$
with $\alpha=2+{\sqrt{5}}$ and $\Delta=20$. The function of $n$ in the left is decreasing and the function with $n$ in the right is increasing, and the inequality holds at $n=1$ (the left--hand side there is $<18.5$ and the right side is $>19.5$), so it holds for all $n\ge 1$. Thus, $\kappa\ge 2$ for $r=4$ as well.
\end{proof}

\subsection{The case $n\le 100$}

We first seek bounds on $r$. Having the bounds in $r$ and $n$, we get bounds on $x$ using Lemma \ref{lem:boundsonx}. Finally, for a fixed $r$ we use Baker-Davenport on estimate Eq. \eqref{Ppx} to lower $x$. The hope is that in all cases $x\le 100$, a case which has already been treated. 

We prove the following result.

\begin{lemma}
When $n\le 100$, we have $r\le 1.5\times 10^6$. 
\end{lemma}

\begin{proof}
We assume $r>10^6$. Then $x>\kappa r^2\log r/1.01\ge nr^2\log r/2.02$ by Lemma \ref{Luca1} and Lemma \ref{lem:better}. We go through the three possibilities of Lemma \ref{lem:boundsonx}. In case (i), we get
$$
r^2\log r\le 1.01\times 3.3\times 10^6 (\log (r+1))^2,
$$
which gives $r<5500$, a contradiction. Assume we are in case (ii). Then 
\begin{equation}
\label{eq:interm}
x+1<1.38\times 10^6 \times 1.001 (\log(r+1))^2 \log\left(\frac{0.3(n+1)(x+1)^2}{(\log(r+1))^2}\right).
\end{equation}
The factor $1.001$ is an upper bound on the factor $1+1/(r\log r)$ which is valid since $r$ is large. 
Now 
$$
x+1>x>\frac{nr^2\log r}{2.02}>nr^2 \log(r+1) \left(\frac{\log r}{2.02\log(r+1)}\right)>\frac{nr^2\log(r+1)}{2.03},
$$
where the last inequality holds since $r>10^6$. Put $y:=(x+1)/(n(\log(r+1))$. Then the above inequality is $y>r^2/2.03$. Inequality Eq. \eqref{eq:interm} can be rewritten in terms of $y$ as 
$$
y<1.38\times 10^6 \log(r+1) \times 1.001 \left(\log(0.3n^2(n+1) y^2)\right)^2<1.39\times 10^6 \log(r+1) (2\log y+\log(30030))^2.
$$
We look at the function
$$
f(y):=\frac{y}{(2\log y+\log(30030))^2}.
$$
Its derivative is 
$$
\frac{2\log y+\log(30030)-4}{(2\log y+\log(30030))^3}>0,
$$
so our function $f(y)$ is increasing. Since $f(y)<1.39\times 10^6\log(r+1)$, and $y>r^2/2.03$, it follows that $f(r^2/2.03)<1.39\times 10^6 \log(r+1)$. This gives
$$
r^2<1.39\times 2.03\times 10^6(\log(r+1)) (2\log(r^2/2.03)+\log(30030))^2,
$$
which gives $r<370000$, a contradiction. Assume we are in case (iii). We use the same substitution $y:=(x+1)/(n(\log(r+1))$. We then  get
$$
y<1.64\times 10^6 (\log(r+1))^2 (\log(0.12 n^2 y^2))^2\le 1.64\times 10^6 (\log(r+1))^2(2\log y+\log(12))^2.
$$
The function $g(y):=y/(2\log y+\log(12))^2$ is also increasing, so we deduce that
$$
r^2<2.03\times 1.64\times 10^6 (\log(r+1))^2 (2\log(r^2/2.03)+\log(12))^2,
$$
and this gives $r<1.5\times 10^6$. 
\end{proof}

Having bounds on $n$ and $r$, inequalities (i), (ii) and (iii) from Lemma \ref{lem:boundsonx} become
\begin{eqnarray*}
x & < & 3.3\times 10^6 (\log(1.5\times 10^6))^2<7\times 10^8;\\
x & < & 1.39 \times 10^6\times 100 (\log(1.5\times 10^6))^2 \left(\log(0.3\times 101 (x+1)^2)\right)^2<3.7\times 10^{10} \log(30.3(x+1))^2,\\
x & < & 1.64\times 10^6\times 100\times (\log(1.5\times 10^6))^3\left(\log(0.12(x+1)\right)^2<6.5\times 10^{11} \log(0.12(x+1))^2.
\end{eqnarray*}
Any one of these inequalities implies that $x<3\times 10^{15}$. Now we do Baker-Davenport on estimates Eq. \eqref{Ppx} for $n\in [2,100]$, $r\in [3,1500000]$, and $x<3\times 10^{15}$. This also gives $m<3\times 10^{17}$ via inequality Eq. \eqref{Prob6}. We return to inequality \eqref{Ppx} and rewrite it as follows.
\begin{align}\label{walax}
\left|x\dfrac{\log U_{n+1}}{\log\alpha} -m + \dfrac{\log\left(\sqrt{r^2+4}\right)}{\log \alpha}\right|< \dfrac{2.2}{r^x\log \alpha}.
\end{align}
Then, we apply Lemma \ref{Dujjella} on Eq. \eqref{walax} with the data:
\begin{align*}
M:=3\times 10^{17}, \quad \tau:=\dfrac{\log U_{n+1}}{\log\alpha}, \quad \mu:=\dfrac{\log\left(\sqrt{r^2+4}\right)}{\log \alpha}, \quad A:=\dfrac{2.2}{\log \alpha}, \quad \text{and}\quad B:=r.
\end{align*}
A computer search in \textit{Mathematica} reveals that $ x\le 81 $, which is a contradiction. This computation lasted 16 hours on a cluster of four 16 GB RAM computers.
\subsection{The case $n> 100$}

Estimate Eq. \eqref{Ppx} together with estimate Eq. \eqref{eq:Un} give
$$
|\Gamma|=|m\log\alpha-\log({\sqrt{r^2+4}}-x((n+1)\log \alpha-\log ({\sqrt{r^2+4}}+\zeta'')|<\frac{2.2}{r^x},
$$
which implies, via estimate Eq. \eqref{eq:zeta''}, that
\begin{equation}
\label{eq:bound2}
|(m-x(n+1))\log \alpha+(x-1)\log ({\sqrt{r^2+4}})|<\frac{2.2}{r^x}+\frac{1.51 x}{\alpha^{2n+2}}.
\end{equation}
If $r\ge 4$ then $\kappa>0$, so by Lemma \ref{Luca1}, we have $\kappa\ge 2$. If $\kappa\ge 3$, then 
$$
x\ge\frac{ \kappa (r^2\log r+1)}{1+5/r}>\frac{n4^2\log 4}{4.5}>2n+2.
$$
The same conclusion holds for $\kappa=2$ since then $x\ge r^{\min\{2,n-3\}}\ge 4^{\min\{2,n-3\}}\ge 2n+2$ for all $n\ge 2$.  
We thus get
\begin{equation}
\label{eq:bound3}
|(m-x(n+1))\log \alpha+(x-1)\log ({\sqrt{r^2+4}})|<\frac{4}{r^{2n+2}}+\frac{1.51 x}{\alpha^{2n+2}}<\frac{1.51x+2.2}{\alpha^{2n+2}}.
\end{equation}
We prove the following lemma. 

\begin{lemma}
\label{lemalphaton}
For $n>100$, we have $r^{n-2}>x$. 
\end{lemma}

Note that Lemmas \ref{Luca1} and \ref{lemalphaton} show that the case $\kappa=2$ cannot occur for $r\ge 4$ provided that $n> 100$.

\begin{proof}
Assume $x\ge r^{n-2}$. We use the bounds given by Lemma \ref{lem:boundsonx} on $x$. In case (i), we get
$$
r^{99}\le r^{n-2}\le 3.3\times 10^6 (\log(r+1))^2.
$$
The function $y\mapsto y^{99}/(\log(y+1))^2$ is increasing for all $y\ge 3$ as one can check by computing its derivative. Thus, if the above inequality holds for $r$, it should hold $r=3$ as well, which is false. In case (ii), we have 
$$
x<1.38\times 10^6 n\left(1+\frac{1}{r\log r}\right)(\log(r+1))^2 \left(\log\left(\frac{0.3(n+1)(x+1)^2}{(\log(r+1))^2}\right)\right)^2.
$$
The expression $1+1/(r\log r)$ is smaller than $1.304$ at $r=3$. Since $1.38\times 1.304<1.8$, we get
$$
x<1.8\times 10^6 (\log(r+1))^2 \left(\log\left(\frac{0.3(n+1)(x+1)^2}{(\log(r+1))^2}\right)\right)^2.
$$
If $x\le n$, then we get $r^{n-2}\le n$ for $r\ge 3$ and $n\ge 101$, which is false. Thus, $n<x$, so we  may use $0.3(n+1)<n<x+1$ to get
\begin{eqnarray*}
x+1 & < & 1.8\times 10^6 n (\log(r+1))^2 (\log((x+1)^3))^2\\
& < & 1.8\times 3^2 \times 10^6 (\log(r+1))^2(\log(x+1))^2\\
& < & 1.7\times 10^7 (\log(r+1))^2 (\log(x+1))^3.
\end{eqnarray*}
The function $(x+1)/(\log(x+1))^3$ is increasing for $x+1>e^3$, which is the case for us since $x\ge r^{99}\ge 3^{99}$. Hence, the above inequality should hold for $x+1$ replaced by $r^{n-2}$, and that gives
$$
r^{n-2}\le 1.7\times 10^7 (\log(r+1))^2 (\log(r^{n-2}))^2<1.7\times 10^7 n^2 (\log(r+1))^4.
$$
Since $r^{n/3}\ge 3^{n/3}>n$ holds for $n>100$, we get that
$$
r^{n/3-2}<1.7\times 10^7 (\log(r+1))^4,
$$
so
$$
r^{95}\le r^{n-6}<\left(1.7\times 10^7 (\log(r+1))^2\right)^3<5\times 10^{21} (\log(r+1))^{12}.
$$
The function $y\mapsto y^{95}/(\log(y+1))^6$ is increasing for $y\ge 3$, so the last inequality should hold also for $r$ replaced by $3$, which is false. A similar argument works if $x$ is in  case (iii) of Lemma \ref{lem:boundsonx}.
We don't give further details. 
\end{proof}

From Lemma \ref{lemalphaton} we conclude that if $r\ge 4$, then inequality Eq. \eqref{eq:bound3} leads to
\begin{equation}
\label{eq:rton}
|(m-x(n+1))\log \alpha+(x-1)\log ({\sqrt{r^2+4}})|<\frac{1}{r^n}.
\end{equation}
We put
\begin{eqnarray*}
\Gamma_2:= (x-1)\log \left(\sqrt{r^2+4}\right)-((n+1)x-m)\log\alpha.  
\end{eqnarray*}
We apply Theorem \ref{Mignotte2} to find a lower bound on $\log |\Gamma_2|$ with the data:
\begin{eqnarray}
t:=2, \quad \gamma_1:=\sqrt{r^2+4}, \quad \gamma_2:=\alpha, \quad b_1:=x-1, \quad b_2:=m-(n+1)x.
\end{eqnarray}
Since $ \gamma_1, \gamma_2 \in\mathbb{Q}(\alpha) $, we take again $ \mathbb{K}:= \mathbb{Q}(\alpha)$ with degree $ {\mathcal D}:=2 $.
The fact that $ \gamma_1 $ and $ \gamma_2 $ are multiplicatively independent has already been checked. 
 
We take 
\begin{eqnarray*}
\max\left\{h(\gamma_1), \dfrac{|\log\gamma_1|}{2}, \dfrac{1}{2}\right\}=\dfrac{1}{2}\log (r^2+4)<\log(r+1):=\log B_1,
\end{eqnarray*}
and 
\begin{eqnarray*}
\max\left\{h(\gamma_2), \dfrac{|\log\gamma_2|}{2}, \dfrac{1}{2}\right\}=\dfrac{1}{2}\log \alpha<\frac{1}{2}\log(r+1):=\log B_2.
\end{eqnarray*}
Thus,
\begin{eqnarray*}
b^{\prime}:=\dfrac{|b_1|}{D\log B_2}+\dfrac{|b_2|}{D\log B_1} = \dfrac{x-1}{2\log (r+1)}+\dfrac{|(n+1)x-m|}{\log(r+1)}<\frac{2.5x}{\log(r+1)},
\end{eqnarray*}
where we used the fact that $m-(n+1)x\in [-2x+2,2]$. 
By Theorem \ref{Mignotte2}, we get that
\begin{equation}
\label{eq:lowGamma2}
\log|\Gamma_2|>-195\left(\max\left\{\log \left(\frac{2.5x}{\log r}\right), 10.5\right\}\right)^{2} (\log(r+1))^2.
\end{equation}
We want an upper bound on $r$. So, assume $r\ge 10^6$. If $ \log ({(2.5x)}/{\log(r+1)}) < 10.5$, then
$$
x<\frac{e^{10.5} \log(r+1)}{2.5}<15000 \log(r+1).
$$
Since $x>nr^2\log r/(2+10/r)\ge 101 r^2\log r/2.02=50r^2\log r$, we get
$$
50 r^2\log r \le x\le 15000 (\log(r+1)),
$$
so 
$$
r^2<\left(\frac{15000}{50}\right) \left(\frac{\log(r+1)}{\log r}\right)<300\times 1.01=303,
$$
so $r\le 17$, a contradiction. 

Assume next that $\log(2.5 x/\log(r+1))>10.5$. We then get
$$
\log |\Gamma_2|>-195 (\log(r+1))^2 \left(\log\left(\frac{2.5 x}{\log(r+1)}\right)\right)^2.
$$
Comparing the above inequality with estimate Eq. \eqref{eq:rton}, we get
$$
n\log r<195 (\log(r+1))^2 \left(\log\left(\frac{2.5 x}{\log(r+1)}\right)\right)^2,
$$
so 
\begin{eqnarray*}
n & < & 195 \left(\frac{\log(r+1)}{\log r}\right) \log(r+1) \left(\log\left(\frac{2.5 x}{\log(r+1)}\right)\right)^2\\
& < & 196 \log(r+1) \left(\log\left(\frac{2.5 x}{\log(r+1)}\right)\right)^2,
\end{eqnarray*}
where we used the fact that $r\ge 10^6$ so $\log(r+1)/\log r<1.0001$. 

We now use the bounds on $x$ given by Lemma \ref{lem:boundsonx}. In case (i), we have
$$
50 r^2\log r\le x\le  3.3\times 10^6 (\log(r+1))^2,
$$
so 
$$
r^2<\left(\frac{3.3\times 10^6}{50}\right)\times \left(\frac{\log (r+1)}{\log r}\right) \log(r+1)<6.7\times 10^4 \log(r+1),
$$
so $r<670$, a contradiction. In case (ii), we have
$$
x+1<1.39\times 10^6 n (\log(r+1))^2 \left(\log\left(\frac{0.3(n+1)(x+1)^2}{(\log(r+1))^2}\right)\right)^2.
$$
In case $(x+1)/(\log(r+1))\le n$, we have 
$$
x+1<196 (\log(r+1))^2 \left(\frac{2.5 x}{\log(r+1)}\right)^2.
$$
So, putting $y:=x/\log(r+1)$, we get $y<196\log(r+1) \log(2.5 y)^2$. Note that 
$$
y=\frac{(x+1)}{\log(r+1)}>\frac{nr^2\log r}{2.01\log(r+1)}\ge \frac{101 r^2}{2.01(\log(r+1)/\log r)}>\frac{101 r^2}{2.02}=50r^2.
$$
In the above inequalities we used the fact that $r\ge 10^6$. 
The function $y\mapsto y/\log(2.5 y)^2$ is increasing for $y>e^2/2.5$, which is our case. Hence, the inequality $y<196\log(r+1) \log(2.5 y)^2$ should hold with 
$y$ replaced by $50r^2$, which yields $50r^2<196 \log(r+1) \log(2.5\times 50 r^2)^2$ and gives $r<50$, a contradiction. Thus, $n<(x+1)/\log(r+1)$. Since $0.3(n+1)<n$, we conclude that 
$$
x+1<1.39\times 10^6 (\log(r+1))^2 \left(196 \log(r+1) \left(\log\left(\frac{2.5 x}{\log(r+1)}\right)\right)^2\right)\left(\log\left(\left(\frac{(x+1)}{\log(r+1)}\right)^3\right)\right)^2.
$$
Putting again $y:=(x+1)/\log(r+1)$, we get
$$
y<1.39\times 10^6\times 196 \times (\log(r+1))^2 \log(2.5 y)^2(3\log y)^2<2.46\times 10^9 (\log(2.5y))^4.
$$
The function $y\mapsto y/(\log(2.5y)^4$ is increasing for our range for $y>50r^2>50\times 10^{16}$, so we get that the above inequality should hold by replacing $y$ by $50r^2$. Thus,
$$
50r^2<2.46\times 10^9 (\log(r+1))^2 (\log(2.5 \times 50 r^2))^4,
$$
so $r<2.6\times 10^8$. A similar argument holds when $x$ is in case (iii). In that case, we may again suppose that $n<(x+1)/\log(r+1)$. We get
$$
x+1<1.64\times 10^6 (\log(r+1))^3\left(196\log(r+1)\left(\log\left(\frac{2.5 x}{\log(r+1)}\right)\right)^2\right)\left(\log\left(\frac{0.12(x+1)}{(\log(r+1))^2}\right)\right)^2,
$$
so
$$
y<3.3\times 10^8 (\log(r+1))^3 (\log(2.5 y))^2(\log(0.12 y))^2.
$$
Imposing that the above inequality holds for $y$ replaced by $50 r^2$, we get
$$
50 r^2<1.3\times 10^9 (\log(r+1))^3 (\log(2.5\times 50 r^2))(\log(0.12 \times 50r^2))^2,
$$
so $r<4.3\times 10^8$. 

To summarise, we have proved the following. 

\begin{lemma}
If $r\ge 4$, then $r<4.3\times 10^8$. 
\end{lemma}

Having bounds for $r$ it is easy to find bounds for $x$. For example,
$$
n<
195\log(r+1)(\log(2.5 x))^2<195\log(4.3\times 10^8)(\log(2.5 x))^2<3900(\log(2.5 x))^2.
$$
Next, if $x$ is in case (i), then
$$
x<3.3\times 10^6 (\log(r+1))^2<3.3\times 10^6 (\log(4.3\times 10^8))^2<1.4\times 10^9.
$$
If $x$ is in case (ii), then
\begin{eqnarray*}
x & < & 1.39\times 10^6 (\log(4.3\times 10^8)^2 (3900(\log(2.5 x))^2)(\log(0.3\times 3900 (\log(2.5 x))^2(x+1)^2))^2\\
& < & 2.2\times 10^{12} (\log(2.5 x))^2(\log(1200(x+1)^2\log(2.5x)))^2,
\end{eqnarray*}
which gives $x<5\times 10^{19}$. Finally, if $x$ is as in case (iii), then
\begin{eqnarray*}
x &< & 1.64\times 10^6(\log(4.3\times 10^8)^3\times 3900(\log(2.5x))^2\times (\log(0.12(x+1)))^2\\
& < & 5.1\times 10^{13} (\log(2.5x))^2(\log(0.12(x+1)))^2,
\end{eqnarray*}
so $x<3\times 10^{20}$. 

Thus, $r<4.3\times 10^8$ and $x<3\times 10^{20}$. Inequality Eq. \eqref{eq:rton} gives that 
\begin{align}\label{wala}
\left|\frac{\log {\sqrt{r^2+4}}}{\log \alpha}-\frac{x(n+1)-1}{x-1}\right|<\frac{1}{r^n (x-1)\log \alpha}<\frac{1}{16 (x-1)^2},
\end{align}
where for the last inequality we used that $r^n=r^2 r^{n-2}\ge 16x$ by Lemma \ref{lemalphaton}. In particular, the ratio $(x(n+1)-1)/(x-1)$ is a convergent of $\log {\sqrt{r^2+4}}/\log \alpha$. 
Since $x<3\times 10^{20}< F_{100}$, it follows that $(x(n+1)-1)/(x-1)=p_k/q_k$ for some $k\in [0,99]$.  So, we apply Lemma \ref{Legendre} on Eq. \eqref{wala} with the data:
\begin{align*}
M:=3\times 10^{20}, \quad \tau:=\frac{\log {\sqrt{r^2+4}}}{\log \alpha}, \quad u:=x(n+1)-1,  \quad \text{and} \quad v:=x-1.
\end{align*}
With the help of a computer search in {\it Mathematica}, we checked all these possibilities over all the values for $4\le r\le 4.3\times 10^8$ and found that $ n\le 46 $, which is a contradiction. This computation lasted 6 hours on a cluster of four 16 GB RAM computers.

\subsection{The case $r=3$}

The case $r=3$ is special since there we don't know that $\kappa>0$ so some of the inequalities used for the case $r\ge 4$ do not apply. In the case $n\le 100$, Lemma \ref{lem:boundsonx} 
gives
$$
x<3.3\times 10^6\times (\log 4)^2<6.4\times 10^6,
$$
or 
\begin{eqnarray*}
x & < & 1.38\times 10^8 \times 100 \left(1+\frac{1}{3\log 3}\right) (\log 4)^2 \left(\log\left(\frac{0.3\times 101 (x+1)^2}{(\log 4)^2}\right)\right),\\
& < & 3.5\times 10^{10} (\log(16(x+1)^2))^2,
\end{eqnarray*}
which gives $x<1.7\times 10^{14}$, or 
\begin{eqnarray*}
x & < & 1.63\times 10^6 \times 100 \left(1+\frac{1}{3\log 3}\right)(\log 4)^3 \left(\log\left(\frac{0.12 (x+1)}{(\log 4)^2}\right)\right)^2\\
& < & 5.7\times 10^8 \left(0.07(x+1)\right)^2,
\end{eqnarray*}
which gives $x<3.3\times 10^{11}$. 
Now we can do Baker-Davenport on estimate Eq. \eqref{Ppx} and get better bounds on $x$. In case $n>100$,  estimate Eq. \eqref{eq:bound2} together with Lemma \ref{lemalphaton} 
hold and give 
\begin{equation}
\label{eq:bound4}
|(m-x(n+1))\log \alpha+(x-1)\log ({\sqrt{r^2+4}})|<\frac{2.2}{r^x}+\frac{1.51 x}{\alpha^{2n+2}}<\frac{2.2}{r^x}+\frac{1.51}{r^{n+2}}<\frac{1}{3^{\min\{n-1,x-1\}}}.
\end{equation}
We keep the notation $r$ and $\alpha$ although this section only applies to $r=3$ for which $\alpha=(3+{\sqrt{13}})/2$.  Put $\ell:=\min\{n-1,x-1\}$. The lower bound estimate Eq. \eqref{eq:lowGamma2} still applies and gives
$$
\ell\log 3<-\log |\Gamma_2|<195\left(\max\left\{\log \left(\frac{2.5x}{\log 4}\right), 10.5\right\}\right)^{2} (\log 4)^2.
$$
When $\ell=n-1$, we get
$$
(n-1)\log 3<375  \left(\max\left\{\log \left(\frac{2.5x}{\log 4}\right), 10.5\right\}\right)^{2}.
$$
In case the maximum above is $10.5$ we get $x<16000$ and $n<40000$. If the maximum above is not at $10.5$, we then get
$$
n<\frac{375}{\log 3} \left(\log\left(\frac{2.5 x}{\log 4}\right)\right)^2+1<350(\log(2x))^2.
$$
Going via the possibilities (i), (ii), (iii), we get 
\begin{eqnarray*}
x & < & 3.3\times 10^6 (\log 4)^2<7\times 10^6;\\
x & < & 1.38\times 10^6\left(350 (\log(2x))^2\right) \left(1+\frac{1}{3\log 3}\right)(\log 4)^2 \left(\log\left(\frac{0.3\times 351 (\log 2x)^2 (x+1)^2}{(\log 4)^2}\right)\right)^2\\
& < & 1.3\times 10^9 (\log 2x))^2(\log(55(x+1)^2))^2,
\end{eqnarray*}
which gives $x<2\times 10^{16}$, or finally
\begin{eqnarray*}
x & < & 1.63\times 10^6 (350 (\log(2x))^2)\left(1+\frac{1}{3\log 3}\right)(\log 4)^3 \left(\log\left(\frac{0.12(x+1)}{(\log 4)^2}\right)\right)^2\\
& < & 2\times 10^9 (\log(2x))^2(\log(0.07(x+1)))^2,
\end{eqnarray*}
which gives $x<3.1\times 10^{15}$. So, in all instances, $x<2\times 10^{16}$, and now 
$$
n<350 (\log(2x))^2<6\times 10^5.
$$
Since Lemma \ref{lemalphaton} still applies, it follows that estimate Eq. \eqref{eq:bound4} gives
$$
\left|\frac{\log{\sqrt{r^2+4}}}{\log \alpha}-\frac{x(n+1)-m}{x-1}\right|<\frac{1}{3^{n-1}(x-1)(\log 3)}<\frac{1}{3(x-1)^2},
$$
so again $(x(n+1)-m)/(x-1)=p_k/q_k$ is a convergent of $\log {\sqrt{r^2+4}}/\log \alpha$ with $x<2\times 10^{16}<F_{80}$, so $k\in [0,\ldots,16]$. So, everything works fine if $\ell=n-1$.

In case $\ell=x-1$, one gets
$$
(x-1)\log 3<375 \left(\max\left\{\log \left(\frac{2.5x}{\log 4}\right), 10.5\right\}\right)^{2},
$$
which gives $x<5\times 10^4$. And one wonders how one should finish it off. We can expand another linear form in logarithms which is small, or we may recall the following main theorem 
from \cite{Biluetal}. 

\begin{theorem}
Assume that $s\not\in \{1,2,4\}$ is minimal such that $U_m\mid U_{n+1}^s-U_n^s$. Then $m<20000s^2$.
\end{theorem}

In our instance, since $U_m=U_n^x+U_{n+1}^x$, one checks that the minimal $s$ is exactly $2x$. Thus, $m<80000x^2$, and since $m>(n-1)x$ by estimate \eqref{Prob6}, we get $n\le 80000x<4\times 10^9$. 
Thus, $n<4\times 10^9$ and $x<5\times 10^4$. It is still a large range and we need to reduce it. 

We consider the element
\begin{eqnarray*}
y:=\dfrac{x}{\alpha^{2n}}.
\end{eqnarray*}
Lemma \ref{lemalphaton} together with the fact that $\alpha>r$ implies that
\begin{eqnarray}\label{Prob14}
y< \dfrac{1}{\alpha^{n}},
\end{eqnarray}
where the last inequality holds for all $ n> 100 $. Now, we write
\begin{eqnarray*}
U_n^x = \dfrac{\alpha^{nx}}{(r^2+4)^{x/2}}\left(1-\dfrac{(-1)^{n}}{\alpha^{2n}}\right)^{x},
\end{eqnarray*}
and
\begin{eqnarray*}
U_{n+1}^x = \dfrac{\alpha^{(n+1)x}}{(r^2+4)^{x/2}}\left(1-\dfrac{(-1)^{n+1}}{\alpha^{2(n+1)}}\right)^{x}.
\end{eqnarray*} 
If $ n $ is odd, then
\begin{eqnarray*}
1< \left(1-\dfrac{(-1)^{n}}{\alpha^{2n}}\right)^{x}= \left(1+\dfrac{1}{\alpha^{2n}}\right)^{x}< e^{y}<1+2y,
\end{eqnarray*}
because $ y $ is very small. On the other hand, if $ n $ is even, then
\begin{eqnarray*}
1> \left(1-\dfrac{(-1)^{n}}{\alpha^{2n}}\right)^{x}= \exp\left(x\log \left(1-\dfrac{1}{\alpha^{2n}}\right)\right)^{x}> e^{-2y}>1-2y,
\end{eqnarray*}
because $ y $ is very small. Thus, the following inequalities hold in both cases,
\begin{eqnarray*}
\left|U_n^x-\dfrac{\alpha^{nx}}{(r^2+4)^{x/2}}\right|< \dfrac{2y\alpha^{nx}}{(r^2+4)^{x/2}},
\end{eqnarray*}
and 
\begin{eqnarray*}
\left|U_{n+1}^x-\dfrac{\alpha^{(n+1)x}}{(r^2+4)^{x/2}}\right|< \dfrac{2y\alpha^{(n+1)x}}{(r^2+4)^{x/2}}.
\end{eqnarray*}
Now, we return to \eqref{Prob2} and rewrite it as
\begin{align*}
\dfrac{\alpha^{m}-\beta^{m}}{(r^2+4)^{1/2}}&=U_m = U_n^x+U_{n+1}^x\\ &= \dfrac{\alpha^{nx}}{(r^2+4)^{x/2}}+\dfrac{\alpha^{(n+1)x}}{(r^2+4)^{x/2}}+\left(U_n^2-\dfrac{\alpha^{nx}}{(r^2+4)^{x/2}}\right)+\left(U_{n+1}^x - \dfrac{\alpha^{(n+1)x}}{(r^2+4)^{x/2}}\right),
\end{align*}
or
\begin{align*}
\left|\dfrac{\alpha^{m}}{(r^2+4)^{1/2}}-\dfrac{\alpha^{nx}(1+\alpha^{x})}{(r^2+4)^{x/2}}\right| &=\left|\dfrac{\beta^{m}}{(r^2+4)^{1/2}} + \left(U_n^x-\dfrac{\alpha^{nx}}{(r^2+4)^{x/2}}\right)+\left(U_{n+1}^x -\dfrac{\alpha^{(n+1)x}}{(r^2+4)^{x/2}}\right)\right|\\
&< \dfrac{1}{\alpha^{m}}+ \left|U_n^x-\dfrac{\alpha^{nx}}{(r^2+4)^{x/2}}\right|+\left|U_{n+1}^x-\dfrac{\alpha^{(n+1)x}}{(r^2+4)^{x/2}}\right|\\ 
&<\dfrac{1}{\alpha^{m}}+2y\left(\dfrac{\alpha^{nx}(1+\alpha^{x})}{(r^2+4)^{x/2}}\right).
\end{align*}
Multiplying both sides of the above inequality by $ \alpha^{-(n+1)x}(r^2+4)^{x/2} $, we obtain that
\begin{eqnarray}\label{kazi1}
\left|\alpha^{m-(n+1)x}(r^2+4)^{(x-1)/2}-(1+\alpha^{-x})\right|< \dfrac{(r^2+4)^{x/2}}{\alpha^{m+(n+1)x}}+2y(1+\alpha^{-x})< \dfrac{1}{2\alpha^{n}}+3y< \dfrac{3}{\alpha^{n}},
\end{eqnarray}
so we may divide both sides of it by $1+\alpha^{-x}$ and get 
$$
\left|\alpha^{m-(n+1)x}(r^2+4)^{(x-1)/2}(1+\alpha^{-x})^{-1}-1\right|<\frac{3}{\alpha^n}.
$$
Since $n>100$, the left--hand side is small (say smaller than $1/2$), so we can pass to a logarithmic form and get
\begin{equation}
\label{eq:last}
|(m-(n+1)x)\log \alpha+(x-1)\log ({\sqrt{r^2+4}})-\log(1+\alpha^{-x})|<\frac{6}{\alpha^n}.
\end{equation}
For us, the parameter $x$ is in $[3,50000]$. Given $x$, we have $m<2\times 10^{14}$. So, this is a suitable inequality to apply Baker-Davenport to. To do so, we rewrite  Eq. \eqref{eq:last} as
\begin{align}\label{walay}
\left|((n+1)x-m)\dfrac{\log \alpha}{\log \left(\sqrt{r^2+4}\right)}-(x-1)+\dfrac{\log (1+\alpha^{-x})}{\log \left(\sqrt{r^2+4}\right)}\right|< \dfrac{6}{\alpha^{n}\log \left(\sqrt{r^2+4}\right)}.
\end{align}
Then, we apply Lemma \ref{Dujjella} on Eq. \eqref{walay} with the data:
\begin{align*}
M:=2\times 10^{14}, \quad \tau:=\dfrac{\log \alpha}{\log \left(\sqrt{r^2+4}\right)}, \quad \mu:=\dfrac{\log (1+\alpha^{-x})}{\log \left(\sqrt{r^2+4}\right)}, \quad A:=\dfrac{6}{\log \left(\sqrt{r^2+4}\right)}, \quad B:=\alpha.
\end{align*}
A computer search in {\it Mathematica} reveals that $ n\le 94 $, which is the final contradiction. This computations lasted a few hours on a cluster four 16 GB RAM computers. \qed

\section*{acknowledgements}
M.~D. was supported by the Austrian Science Fund (FWF) grants: F5510-N26 -- Part of the special research program (SFB), ``Quasi-Monte Carlo Methods: Theory and Applications'' and W1230 --``Doctoral Program Discrete Mathematics''. F.~L. was supported by grant RTNUM19 from CoEMaSS, Wits, South Africa. This paper was written when both authors visited the Max Planck Institute for Mathematics Bonn, in 2019. They thank this institution for the hospitality and the fruitful working environment.


\begin{thebibliography}{10} 

%

\bibitem{BHV} Yu. Bilu, G. Hanrot, and P. M. Voutier,
\newblock {E}existence of primitive divisors of Lucas and Lehmer numbers. With an appendix of M. Mignotte,]
\newblock{\em J. Reine Angew. Math.} {\bf 539} (2001), 75--122.

\bibitem{Biluetal} Yu. Bilu, T. Komatsu, F. Luca, A Pizarro-Madariaga, and P. St\u anic\u a, {\it On a divisibility relation for Lucas sequences}, {\em J. Number Theory} {\bf 163} (2016), 1--18.

%
%
%

\bibitem{Car} R. D. Carmichael,
\newblock {O}n the numerical factors of the arithmetic forms $\alpha^n\pm \beta^n$ 
\newblock {\em Ann. of Math. (2)} {\bf 15} (1913), 30--70.

\bibitem{Gomez} C. A. G\'omez Ru\'{\i}z and F. Luca, {An exponential Diophantine equation related to the sum of powers of two consecutive $ k $--generalized Fibonacci numbers}, \newblock{\em Colloq. Math.} {\bf 137} (2014), 171--188.

\bibitem{dujella98}
A. Dujella and A. Peth{\H o},
\newblock {A} generalization of a theorem of {B}aker and {D}avenport,
\newblock {\em Quart. J. Math. Oxford Ser. (2)}, {\bf 49} (1998), 291--306.


\bibitem{TK}
T. Koshy,
\newblock {\em Fibonacci and Lucas Numbers with Applications},
\newblock A Wiley-Interscience Publication, 2001.


\bibitem{Laurent:1995}
M. Laurent, M. Mignotte, and Yu. Nesterenko,
\newblock {F}ormes lin{\'e}aires en deux logarithmes et d{\'e}terminants d'interpolation,
\newblock {\em J. Number Theory}, {\bf 55} (1995), 285--321.

\bibitem{Luca1} F. Luca and R. Oyono, {An exponential Diophantine equation related to powers of two consecutive  Fibonacci numbers}, \newblock{\em Proc. Japan Acad.} {\bf 137} Ser. A (2014), 45--50.


\bibitem{Mignotte} M. Mignotte,
\newblock {A} kit for linear forms in three logarithms,  
\newblock {\em Preprint} 2008. 

\bibitem{Luca3} S. E. Rihane, B. Faye, F. Luca, and A. Togb\'e, {An exponential Diophantine equation related to the difference between powers of two consecutive Balancing numbers}, \newblock{\em Ann. Math. Inform.} {\bf 50} (2019), 167--177.

\bibitem{Luca2} S. E. Rihane, B. Faye, F. Luca, and A. Togb\'e, {On the exponential Diophantine equation $ P_n^x+P_{n+1}^x=P_m $}, \newblock{\em Turkish J. Math.} {\bf 43} (2019), 1640--1649.

\end{thebibliography}
 \end{document}